\documentclass[11pt]{amsart}

\usepackage[margin=1in]{geometry}  % set the margins to 1in on all sides
\usepackage{graphicx}              % to include figures
\usepackage{amsmath}               % great math stuff
\usepackage{amsfonts}              % for blackboard bold, etc
\usepackage{amsthm}      % better theorem environments
\usepackage{mathtools}          
\usepackage{color}
\usepackage{tabulary}
\usepackage{caption}
\usepackage{subcaption}
\usepackage{amssymb}
\usepackage{todonotes}
\theoremstyle{plain}
\usepackage{mathptmx}
\usepackage{tikz-cd} 
\usepackage{mathrsfs}
\usepackage{verbatim}
\usepackage{tikz}
\usepackage{enumitem}
\usepackage{hyperref}

\newenvironment{theorem}[2][Theorem]{\begin{trivlist}
\item[\hskip \labelsep {\bfseries #1}\hskip \labelsep {\bfseries #2}]}{\end{trivlist}}

\newtheorem{thm}{Theorem}[section]
\newtheorem{lem}[thm]{Lemma}
\newtheorem{prop}[thm]{Proposition}
\newtheorem{cor}[thm]{Corollary}
\newtheorem{conj}[thm]{Conjecture}
\newtheorem{defn}[thm]{Definition}
\newtheorem{rmk}[thm]{Remark}

\newtheorem{exam}[thm]{Example} 

\DeclareMathOperator{\id}{id}

\DeclareMathOperator{\Stab}{Stab}
\DeclareMathOperator{\Fix}{Fix}

\newcommand{\NN}{\mathbb{N}}      
\newcommand{\ZZ}{\mathbb{Z}}      % for Integers
\newcommand{\QQ}{\mathbb{Q}} 
\newcommand{\RR}{\mathbb{R}}      % for Real numbers
\newcommand{\CC}{\mathbb{C}}     
\newcommand{\DD}{\mathbb{D}}      
\newcommand{\HH}{\mathbb{H}}

\DeclareMathOperator{\SLC}{SL_2(\CC)}
\DeclareMathOperator{\SLR}{SL_2(\RR)}

\DeclareMathOperator{\SLZ}{SL_2(\ZZ)}

\begin{document}

\title[A Sequence of Algebraic Integer Relation Numbers which Converges to $4$]{A Sequence of Algebraic Integer Relation Numbers \\ which Converges to $4$}

\author{Wonyong Jang }
\author{Kyeongro Kim}
\address{Department of Mathematical Sciences, KAIST,
291 Daehak-ro, Yseong-gu, 
34141 Daejeon, South Korea}
\email{jangwy@kaist.ac.kr \& cantor14@kaist.ac.kr}

\maketitle
\begin{abstract}
Let $\alpha \in \RR$ and let 
$$A=\begin{bmatrix} 1 & 1 \\ 0 & 1\end{bmatrix} \ \text{and} \ B_{\alpha} = \begin{bmatrix} 1 & 0 \\ \alpha & 1\end{bmatrix}.$$
The subgroup $G_\alpha$ of $\SLR$ is a group generated by the matrices $A$ and $B_\alpha$. 
In this paper, we investigate the property of the group $G_\alpha.$  We construct a generalization of the Farey graph for the subgroup $G_\alpha.$ This graph determines whether the group $G_\alpha$ is a free group of rank $2$. 
More precisely, the group $G_\alpha$ is a free group of rank $2$ if and only if the graph is tree. 
In particular, we show that if $1/2$ is a vertex of the graph, then $G_\alpha$ is not a free group of rank $2$. 
Using this, we construct a sequence of real numbers so that the sequence converges to $4$ and each number has the corresponding group that is not  a free group of rank $2$. It turns out that the real numbers are algebraic integers.


\smallskip
\noindent \textbf{MSC classes:} 20F65, 37C85, 37E10, 57M60, 20E05.

\end{abstract}
\vspace{1.2cm}

%%%%%%%%%%%%%%%%%%%%%%%%%%%%%%%%%%%%%%%%%%%%%%%%%%%%%%%%%%%%%%%%%%%%%%

\section{Introduction} \label{Introduction}
 Let $\alpha$ be a complex number and let  
 $$A=\begin{bmatrix}
1 & 1 \\ 0 & 1
\end{bmatrix} \ \textnormal{ and } \ B_{\alpha}=\begin{bmatrix}
1 & 0 \\ \alpha & 1
\end{bmatrix}$$ be parabolic matrices in $\SLC.$ We then consider the subgroup of $\SLC$ generated by $A$ and $B_\alpha,$ and  denote this subgroup by $G_\alpha.$ 
The complex number $\alpha$ is called a \textsf{free number} if $G_{\alpha}$ is isomorphic to the free group of rank $2$. Otherwise, we say that the complex number $\alpha$ is a \textsf{relation number}. This terminology was suggested by Kim and Koberda \cite{kim2019non}.

Our main goal is to characterize relation numbers. Indeed, since transcendental numbers are free numbers \cite{MR0003414}, almost all complex numbers are free numbers.
Also, due to Brenner \cite{MR75952} and Sanov \cite{MR0022557}, it is  known that if $\alpha$ is a real number and $|\alpha| \geq 4$, then $\alpha$ is a free number.
 More generally, if $\alpha$ is in the Riley slice of the Schottky space, then the group $G_\alpha$ is free and discrete  \cite{MR1272421}.
For simple descriptions of free numbers, see \cite{MR258975}, \cite{MR94388} and \cite{MR505107}.
However, in the complement of the Riley slice, the characterization of free numbers has not been completed. 
It is thus meaningful to understand the complement of the Riley slice. 

From now on, we will focus on the complement of the Riley slice.
 As mentioned in \cite{MR3621679}, if two numbers $a$ and $b$ in $\CC$ are algebraically conjugate, the corresponding Galois conjugation gives a group isomorphism between $G_a $  and $G_b$. This implies that algebraic free numbers are dense in $\CC$. 
 On the other hand, Ree proved that relation numbers are dense in the unit open disc $\{ z \in \CC : |z| < 1 \}$  \cite{MR142612}.
In particular, the set of relation numbers is dense in $[-4,4] \subset \RR$ \cite{MR94388}.
 
In the rational number case, Beardon \cite{MR1245077}, Tan and Tan \cite{MR1420342} gave some classes of convergent sequences of  rational relation numbers. 
Also Kim and Koberda \cite{kim2019non} found many "simple" sequences of rational relation numbers. More precisely for each $m\in \ZZ$ with $1 \leq m \leq 27$ and  $m\neq24$, the integer $m$  is a \textsf{good numerator}, meaning that the rational number $m/n$ is a relation number for all $n$ with $| m/n | < 4$ (see \cite{MR1370894} and \cite{kim2019non}). 
For more results about rational relation numbers, see \cite{MR2369190}. In particular, one of the most outstanding conjectures is the following.

\begin{conj} 
 If $\alpha \in \QQ$ and $|\alpha|<4$, then $\alpha$ is a relation number.
\end{conj}

 However, it is still unknown whether there exists a sequence of rational relation numbers converging to $3$ or $4$.
Also, it is an open question to determine which value is a limit point of relation numbers in some other classes such as  algebraic integers.
Our main theorem is the following.

\begin{theorem}{\ref{themainthm}}
 There exists a sequence of polynomials $p_n(\alpha)$ satisfying the following:
\begin{itemize}
	\item each polynomial $p_n(\alpha)$ is a monic polynomial of degree $n$ with integer coefficients;
	\item all roots of $p_n(\alpha)$ are distinct real numbers and are relation numbers; and 
	\item let $\alpha_n$ be the maximal root of $p_n(\alpha)$. Then the sequence $\{\alpha_n\}_{n=1}^\infty$ is increasing and converges to $4$.
\end{itemize}
\end{theorem}

 This paper is organized as follows.
In Section $\ref{Preliminaries}$ we introduce basic notions.
In Section $\ref{CircleAction}$ we discuss a circle action of $G_{\alpha}$ in terms of stabilizer subgroups.  
In Section $\ref{TheGeneralizedFareyGraph}$ by using the action of $G_{\alpha}$, we define the generalized Farey graph $\Gamma_{\alpha}$ and prove the following theorem.

\begin{theorem}{\ref{thethm}}
 Let $\alpha \in \RR$. Then $\alpha$ is a relation number if and only if the graph $\Gamma_{\alpha}$ is not tree.
\end{theorem}

This gives a necessary and sufficient condition for relation numbers.
In Section $\ref{orbittestforrelaion}$ by using Theorem \ref{thethm}, we prove the so-called orbit test. This gives us a criterion for relation numbers. In particular, we obtain some rational relation numbers using this test and present them.
Finally, we prove Theorem \ref{themainthm} in Section $\ref{proofofmainthm}$.

\vspace{0.6cm}

\noindent \textbf{Acknowledgment} We are grateful to Sang-hyun Kim and Hyungryul Baik for helpful comments. 
The authors thank Inhyeok Choi and Philippe Tranchida  for careful reading, useful comments, and corrections.
We would also like to thank  Seunghun Lee, Hongtaek Jung, and Donggyun Seo for helpful discussions.
The first author was partially supported by Samsung Science $\And$ Technology Foundation grant No. SSTF-BA1702-01, and the second author was partially supported by the Mid-Career Researcher Program (2018R1A2B6004003) through the National Research Foundation funded by the government of Korea. The authors thank the referee for several helpful comments.

%%%%%%%%%%%%%%%%%%%%%%%%%%%%%%%%%%%%%%%%%%%%%%%%%%%%%%%%%%%%%%%%%%%%%%

\section{Preliminaries} \label{Preliminaries}
\subsection{A historical remark on relation numbers}
 In this section, we review the convention of previous literatures and fix ours.
For $\alpha_1 , \alpha_2 \in \CC$, consider two matrices
$$A_{\alpha_1}=\begin{bmatrix}
1 & \alpha_1 \\ 0 & 1
\end{bmatrix} \textnormal{ and } \ B_{\alpha_2}=\begin{bmatrix}
1 & 0 \\ \alpha_2 & 1
\end{bmatrix},$$ and let 
$$G_{\alpha_1,\alpha_2}=\left< A_{\alpha_1} , B_{\alpha_2} \right> \subset  \SLC.$$
 
 Many previous results have been obtained in the setting $G_{2,\lambda}$ or $G_{\mu,\mu}$.  Moreover, it is known that $G_{a,b}$ is isomorphic to $G_{1,ab} = G_{ab}$ for any nonzero complex numbers $a,b$ \cite{MR94388}.

In this paper, we will write $G_{\alpha}=G_{1,\alpha}$. We say that $\alpha \in \CC$ is $\textsf{a free number}$ if $G_\alpha$ is a free group of rank $2$. Otherwise, we call $\alpha$ a $\textsf{relation number}$. 
As mentioned before, this was first defined by Kim and Koberda \cite{kim2019non}. We note that, in that paper, they use $G_{q}=G_{q,1}$ but we use $G_{\alpha}=G_{1,\alpha}$ for convenience.

\subsection{The circle action of $G_{\alpha}$}\label{A circle action}
 Now we introduce basic definitions and notations. Let $\alpha$ be a real number.  Recall that $\SLC$ acts on the Riemann sphere $\CC\cup \{\infty \}.$ In particular, $\SLR$ acts on the upper half plane $\HH^2.$ 
Now we consider the Cayley transformation $$\phi(z)=i \ \frac{z-i}{z+i}$$ on the Riemann sphere which is a map from the upper half plane $\HH^2$ to the Poincare disk $\DD$. The map $\phi$ allows us to identify $\RR \cup \{ \infty \}$ with $S^1=\partial \DD.$  
In this sense, the group $G_{\alpha}$ acts on the circle $S^1=\RR\cup \{\infty\}$ since $\alpha \in \RR$ and $G_{\alpha}$ is a subgroup of $\SLR.$
 
 Let $M \in \SLR$ and consider the circle action.  We denote the \textsf{fixed point set} of $M$ by $\Fix(M)$, namely $$\Fix(M)=\{ p \in S^1 : M \cdot p = p \}.$$
 When $A$ is a subset in a set $X$ and a group $G$ acts on $X$, we denote the $\textsf{stabilizer subgroup}$ of $A$ by $\Stab_G(A)$,  that is  $$\Stab_G(A)=\{ g \in G : g \cdot A =A \}.$$
If $A=\{ p \}$, then $\Stab_G(p)$ denotes the set $$\{ g \in G : g \cdot p = p \}.$$

\subsection{Some combinatorial notions}
 We summarize basic combinatorial concepts and notations that will appear in Section \ref{TheGeneralizedFareyGraph}.
In this paper, all graphs are simple and undirected so we think of each edge as a two points subset of the vertex set. Let $\Gamma$ be a graph.
For $n \in \NN$, let $P_n$ be the graph defined by vertex set $V=\{ 0 , 1 , \cdots , n \}$ and edge set $E=\left \{ \{ i,i+1 \} : 0 \leq i \leq n-1 \right \}$. A $\textsf{path}$ $P$ in $\Gamma$ is an image of a graph morphism $f : P_n \to \Gamma$. Thus, for a given path $P$, we can express $P$ as a finite sequence of edges $y_1 , \cdots , y_n$. 
Here, $y_i$ is an image of the edge $\{ i-1,i \}$ so we call it the $\textsf{i-th edge}$ of the path $P$. Moreover, the image of $\{ 0 \}$ and $\{ n \}$ are called the $\textsf{starting point}$ of $P$ and the $\textsf{terminal point}$ of $P$, respectively.
We say that a path $P$ has no $\textsf{backtrackings}$ if $y_i \neq y_{i+1}$ for all $i$ in the sequence expression $y_1 , \cdots , y_n$.
If a path $P=y_1 , \cdots , y_n$ has no backtrackings, we define the $\textsf{length}$ of $P$ as $n$.

When a connected graph does not have a cycle graph as a subgraph, we say that the graph is a $\textsf{tree graph}$. The following is a well-known criterion for a tree graph.

\begin{lem} \label{DefTree}
 Let $T$ be a graph. Then $T$ is tree if and only if for any two vertices $v,w$, there exists a unique path without backtracking from $v$ to $w$.
\end{lem}
\begin{proof}
 See Theorem 1.5.1 in \cite{MR2159259}.
\end{proof}

\subsection{Clockwise and anticlockwise maps} \label{Clockwise and anticlockwise maps}

 In Section $\ref{proofofmainthm}$, we will deal with some continuous maps from an open subset $I$ of $\RR$ to the unit circle $S^1=\partial \DD$ on $\CC$. More precisely, they are meromorphic functions restricted to $\RR,$
 and their images are contained in $\RR \cup \{\infty\}$. So we will think of  a map from an open subset $I$ of $\RR$ to $\RR \cup \{ \infty \}$ as  a map from $I$ to $S^1$ under the identification by the Cayley transformation $\phi,$ and vise versa.

In the proof of Theorem \ref{themainthm}, the key objects are (anti)clockwise maps which wind intervals of $\RR$ around $S^1.$ 
Let $I$ be an open subset of $\RR$. 
First, a map $f$ from $I$ to $S^1$ is $\textsf{strictly increasing}$ (or $\textsf{strictly decreasing}$) at $a \in J:= I -f^{-1}(\{\infty\})$ if there is an $\epsilon>0$ such that $(a-\epsilon, a+\epsilon) \subseteq J $ and the map $f|_{(a-\epsilon, a+\epsilon)}$ is strictly increasing (or strictly decreasing respectively). 
Note that if $f$ is a continuous map from  $I$ to $S^1$, then $-1/f$ is also a continuous map from $I$ to $S^1$. 
A continuous map $f$ from $I$ to $S^1$ is $\textsf{anticlockwise}$ at $x_0 \in I$ if one of $f$ or $-1/f$ is strictly increasing at $x_0.$ We call $f$ an $\textsf{anticlockwise map}$ when $f$ is anticlockwise  at all points in $I$. 
Similarly, we can define $\textsf{clockwiseness}$ at a point and a $\textsf{clockwise map}$. 

We give simple properties of anticlockwise and clockwise maps. They will be used in the proof of Lemma \ref{rotationlemma}. The first two lemmas follow easily from the definition, so we only give a rigorous proof of the third lemma. 

\begin{lem} \label{rotationplusconstant}
 Let $I$ be an open interval in $\RR$. For a given anticlockwise (or clockwise) map $f:I \to S^1$and a constant $c_0 \in \RR$,  $g(x):=f(x)+c_0$ is again an anticlockwise (or a clockwise map, respectively). 
\end{lem}
 Here, we allow an open interval to be $(-\infty,a) , (a,\infty)$ or $(-\infty,\infty)=\RR$.

\begin{lem} \label{rotationinverse}
 Let $I$ be an open interval in $\RR$. If $f(x):I \to S^1$ is anticlockwise (or clockwise), then a map $g(x) := 1/f(x)$ is clockwise (or anticlockwise, respectively).
 \end{lem}

\begin{lem} \label{rotationplusx}
 Let $I$ be an open interval in $\RR$. Suppose that $f:I \to S^1$ is anticlockwise. Then $g:I \to S^1$ defined by $g(x):=f(x)+x$ is also anticlockwise.
\end{lem}
\begin{proof}
 Let $a$ be a point in $I.$ First we consider the case where $f(a) \neq \infty$. Then $g(a) \neq \infty$. As the map $f$ is strictly increasing at $a$, $g$ is strictly increasing at $a$ so $g$ is anticlockwise at $a$. 
Next we consider the case where $f(a) = \infty.$ Since $f$ is anticlockwise, 
we can take a real number $\epsilon_0$ with $0<\epsilon_0<1$ so that $-1/f (x)\neq \infty $ for all $x\in (a-\epsilon_0,a+\epsilon_0)$ and the map $-1/f | _{(a-\epsilon_0,a+\epsilon_0)}$ is strictly increasing. 
We choose a real number $M$ so that $M>|a+1|$ and $M>|a-1|.$ Then since $f$ is continuous at $a$ and $f(a)=\infty$, there is a real number $\epsilon_1$ with $0<\epsilon_1<\epsilon_0$ such that $|f(x)| > M$ for all $x\in (a-\epsilon_1 , a+\epsilon_1).$

Now we consider a map $h:I \to S^1$ defined by the equation $h(x)=-1/g(x).$ 
Since $g(a)=\infty$, it is sufficient to show that $h(x)$ is strictly increasing at $a$. 
Since $h(a)=-1/g(a)=0$ and $h$ is continuous at $a$, there is a real number $\epsilon$ with $0<\epsilon<\epsilon_1$ such that $|h(x)|<1$ for all $x\in (a-\epsilon, a+\epsilon).$ 
We claim that the map $h$ is strictly increasing on the open interval $(a-\epsilon, a+\epsilon)$. 

First we focus on the interval $(a, a+\epsilon).$ Fix two points $x_1$ and $x_2$ with $a < x_1 < x_2 < a+\epsilon$.
 Since $f(a) = \infty$ and the map $-1/f(x)$ is strictly increasing on the interval $(a-\epsilon,a+\epsilon),$ we have that  $0=-1/f(a)<-1/f(x_i)$ for all $i\in \{1,2\}.$ 
Hence $f(x_1)<0$ and  $f(x_2) < 0$ since $-1/f(x_i)\neq \infty$ for all $i\in \{1,2\}$. Then as $0<-1/f(x_1) < - 1/f(x_2)$, we can get that $f(x_i)\neq \infty$ for all $i\in \{1,2\}$ and $$f(x_1)<f(x_2)<-M<0.$$
Also $f(x_1)+x_1<f(x_2)+x_2$ since $x_1<x_2$. Note that the map $f(x)+x$ is continuous on $ (a-\epsilon, a+\epsilon)$ and $f(x)+x \neq 0$ for all $x\in (a-\epsilon, a+\epsilon)$ as $|h(x)|<1$ for all $x\in (a-\epsilon, a+\epsilon).$  Thus either 
$$ 0 < f(x_1)+x_1<f(x_2)+x_2 \textnormal{ or } f(x_1)+x_1<f(x_2)+x_2 < 0.$$
However if $0<f(x_1)+x_1<f(x_2)+x_2,$ then it is in contradiction with
$$f(x_1)+x_1 < -M + (a+\epsilon) < - M + a+1< -M+|a+1| < 0.$$ Therefore $ f(x_1)+x_1<f(x_2)+x_2 < 0$ and $f(x_i)+x_i\neq \infty$ for all $i\in \{1,2\}.$ This implies that  $0 < h(x_1) < h(x_2)$ and so $h$ has positive values and is strictly increasing on $(a, a+\epsilon).$
		
Now we consider the map $h$ on the interval $(a-\epsilon, a).$ Fix two points $x_1$ and $x_2$ with $a-\epsilon < x_1 < x_2 < a$.
Since $f(a) = \infty$ and the map $-1/f(x)$ is strictly increasing on the interval $(a-\epsilon,a+\epsilon),$ we have that $ -1/f(x_i) < -1/f(a)=0$ for all $i\in \{1,2\}.$ 
Hence $f(x_1)>0$ and $f(x_2) > 0$ since $-1/f(x_i)\neq \infty$ for all $i\in \{1,2\}.$ Then as $-1/f(x_1) < -1/f(x_2)<0,$ we can get that $f(x_i)\neq \infty$ for all $i\in \{1,2\}$ and
$$0 < M < f(x_1) < f(x_2).$$
Also $f(x_1)+x_1<f(x_2)+x_2$ since $x_1<x_2$. Note that the map $f(x)+x$ is continuous on $ (a-\epsilon, a+\epsilon)$ and $f(x)+x \neq 0$ for all $x\in (a-\epsilon, a+\epsilon)$ as $|h(x)|<1$ for all $x\in (a-\epsilon, a+\epsilon).$  Thus either
$$0 < f(x_1) + x_1 < f(x_2) + x_2  \textnormal{ or } f(x_1)+x_1 < f(x_2) + x_2 < 0.$$
However if $f(x_1)+x_1<f(x_2)+x_2<0,$ then it is in contradiction with
$$f(x_2)+x_2 > M + (a - \epsilon) > M + (a-1) > M - |a-1| > 0.$$Therefore $0< f(x_1)+x_1<f(x_2)+x_2 $ and $f(x_i)+x_i\neq \infty$ for all $i\in \{1,2\}.$ 
This implies that  $ h(x_1) < h(x_2)<0$ and so $h$ has negative values and is strictly increasing on $(a-\epsilon,a).$ Therefore the map $h$ is strictly increasing on $(a-\epsilon,a+\epsilon)$. Thus the map $g$ is anticlockwise. 
\end{proof}

\subsection{Winding numbers} \label{winding}
Now we will define the winding number for continuous maps on $S^1$ modified slightly for our purpose and show some facts about the winding number. Let $f: S^1 \to S^1$ be a continuous map. 
For each point $x\in S^1=\partial \DD,$ we define a quotient map $\pi_x$ from $[0,1]$ to $S^1$ by the equation $\pi_x(s)=xe^{2\pi i s},$ and define a map $f_x: [0,1]\to S^1$ so that $f_x(s)=f\circ \pi_x(s)/f(x)$ for all $s\in [0,1].$ 
 Let $p$ be a map from $\RR$ to $S^1=\partial \DD$ defined by the equation $p(s)=e^{2\pi i s}.$ The map $p$ will be the universal covering map of $S^1.$ 
Now we can define  a map $\tilde{f_x}$ to be the lift of $f_x$ by the covering $p$ such that $\tilde{f_x}(0)=0.$ 
Then we define the $\textsf{winding number}$ of $f$ by $|\tilde{f_x}(1)|$ for some $x\in S^1.$ Note that for any two points $x$ and $y$ in $S^1,$ $\tilde{f_x}(1)$ and $\tilde{f_y}(1)$ are the same integer number. 
Hence the winding number of $f$ is well defined.  We will denote the winding number of $f$ by $w(f).$ 
 
 Let now the map $f:\RR \to  S^1$ be a clockwise or anticlockwise map. 
In this paper,  the map $f$ usually satisfies   $\displaystyle \lim_{x \to -\infty} f(x) = \lim_{x \to \infty} f(x).$ Hence we assume that  $\displaystyle \lim_{x \to -\infty} f(x)$ and $\displaystyle \lim_{x \to \infty} f(x)$ exist and are equal. 
Then we can define the \textsf{extension map} $\overline{f}$ of the map $f$ so that the map $\overline{f}$ is a continuous map on $S^1$ and makes the following diagram commute. 
 \begin{center}
\begin{tikzcd}
\RR \arrow[swap,d,"\phi"]  	\arrow[r,"f"] &S^1\\
S^1 \arrow[swap,ur,"\overline{f}"]
 \end{tikzcd}
\end{center}
where the map $\phi$ is the Cayley transformation. Then we can consider the winding number of $\overline{f}.$
 The following lemma is a convenient tool to calculate the winding number.

\begin{lem} \label{winding-calculuation}
 Let $f : \RR \to S^1$ be clockwise or anticlockwise. Suppose that  $\displaystyle \lim_{x \to -\infty} f(x)$ and $\displaystyle \lim_{x \to \infty} f(x)$ exist and are equal. 
Then the extension map $\overline{f}$ of the map $f$ is a covering map of $S^1.$ Moreover the number of sheets of $\overline{f}$ equals $w(\overline{f}).$ 
\end{lem}
\begin{proof} 
Suppose that $f$ is anticlockwise. For brevity we write $x=\overline{f}(\phi(\infty))\in \partial \DD$ and $h=\overline{f}_{\phi(\infty)}$ Then we consider the lifting map $\tilde{h}.$ Note that the following diagram commutes. 
\begin{center}
\begin{tikzcd}
&&&&\RR \arrow[d,"p"] \\
{(0,1)}\arrow[hookrightarrow]{r}{i} \arrow{drr}{D}&{[0,1]}\arrow[r,"\pi_{\phi(\infty)}"] \arrow[bend left]{urrr}{\tilde{h}} \arrow[bend left=40]{rrr}{h}
&S^1\arrow[r,"\overline{f}"] &S^1 \arrow[r, "R_{1/x}"]& S^1 \\
&&\RR \arrow[u,"\phi"]\arrow[ru, "f"]&&
\end{tikzcd}
\end{center}
where the map $i$ is the inclusion map, the  map $R_{1/x}$ is defined by $R_{1/x}(z)=z/x,$ and  we write $D=\phi^{-1}\circ \pi_{\phi(\infty)} \circ i.$ 
Moreover the map $D$ is a strictly increasing continuous map so it is a homeomorphism. Therefore the map $R_{1/x}\circ f \circ D$ is anticlockwise as the map $f$ is anticlockwise. 
Hence the map $\tilde{h}\circ i$ is strictly increasing since the map $\tilde{h}\circ i$ is a lifting of $R_{1/x}\circ f \circ D=h\circ i$ by the covering map $p$ and the map  $p$ is anticlockwise. 
Therefore the continuous map $\tilde{h}$ is a strictly increasing map from $[0,1]$ to $[0,\tilde{h}(1)]$ since $\tilde{h}(1)$ is a positive integer.

Now we claim that the map $R_{1/x} \circ \overline{f}$ is a covering map.
We write $n=\tilde{h}(1)$ and we define the map $c_n$ on $S^1$ such that  $c_n(z)=z^n$ for all $z\in \partial \DD$, and the map  $\pi^n : [0,n] \to S^1$ such that $\pi^n(s)=\pi_{\phi(1)}(s/n)$ for all $s\in [0,n].$ Then the following diagram also commutes. 
\begin{center}
\begin{tikzcd}
{[0,1]}\arrow[r, "\tilde{h}"] \arrow[d, "\pi_{\phi(\infty)}"]&{[0,n]}\arrow[d, "p"] \arrow[r, "\pi^n"]&S^1\arrow[ld, "c_n"] \\ 
S^1\arrow[r, "R_{1/x}\circ \overline{f}"]   &S^1&
\end{tikzcd}
\end{center}
As the map $\tilde{h}$ is a homeomorphism from $[0,1]$ to $[0,n],$ there is a unique homeomorphism $H$ of $S^1$  which makes the following diagram commute.
\begin{center}
\begin{tikzcd}
{[0,1]}\arrow[r, "\tilde{h}"]\arrow[swap, d, "\pi_{\phi(\infty)}"]& {[0,n]}\arrow[d, "\pi^n"]\\
S^1 \arrow[dotted, swap]{r}{\exists ! H} & S^1
\end{tikzcd}
\end{center}
Therefore the following diagram also commutes.
\begin{center}
\begin{tikzcd}[column sep=small]
S^1 \arrow[rr,"H"] \arrow[swap , rd, "R_{1/x}\circ \overline{f}"]& & S^1\arrow[dl, "c_n"]\\
& S^1&
\end{tikzcd}
\end{center}
Hence as the map $c_n$ is a $n$-fold covering, the map $R_{1/x}\circ \overline{f}$ is a $n$-fold covering. 
Therefore the claim is proved. Thus the claim implies that the map $\overline{f}$ is a $n$-fold covering. Likewise we can prove the case where the map $f$ is clockwise. 
\end{proof}
We end this section by stating simple properties of the winding number. These lemmas follow from  Lemma \ref{winding-calculuation}.

\begin{lem} \label{windingplusconstant}
Let $f$ be a map from $\RR$ to $S^1.$  Suppose that the map $f$ is clockwise or anticlockwise and that $\displaystyle \lim_{x \to -\infty} f(x)$ and $\displaystyle \lim_{x \to \infty} f(x)$ exist and are equal. 
Then for any $c \in \RR,$ a map $g_c: \RR \to \RR \cup \{ \infty \}$ defined by $g_c(x):=f(x)+c$ satisfies that 
$\displaystyle \lim_{x \to -\infty} g_c(x)$ and $\displaystyle \lim_{x \to \infty} g_c(x)$ exist and are equal.
Moreover $w(\overline{f})=w(\overline{g_c})$ for all $c\in \RR.$
\end{lem}

\begin{lem} \label{windinginverse}
Let $f$ be a map from $\RR$ to $S^1.$  Suppose that the map $f$ is clockwise or anticlockwise and that $\displaystyle \lim_{x \to -\infty} f(x)$ and $\displaystyle \lim_{x \to \infty} f(x)$ exist and are equal. Then 
$\displaystyle \lim_{x \to -\infty} g(x)$ and $\displaystyle \lim_{x \to \infty} g(x)$ exist and are equal where the map $g$ is defined by $g(x)=1/f(x)$ as a map from $\RR$ to $\RR\cup \{\infty\}.$ Moreover $w(\overline{f})=w(\overline{g}).$
\end{lem}

%%%%%%%%%%%%%%%%%%%%%%%%%%%%%%%%%%%%%%%%%%%%%%%%%%%%%%%%%%%%%%%%%%%%%%

\section{A Stabilizer Subgroup of the Circle Action} \label{CircleAction}

From now on we discuss the case where $\alpha$ is a positive real number. 
The positivity condition is not restrictive since $G_{\alpha} = G_{-\alpha}$ and $G_0$ is not of rank 2.

\begin{prop} \label{prop4}
Let $\alpha$ be a positive real number. If $\alpha$ is a free number, then for any $p\in S^1,$ the group $\Stab_{G_{\alpha}}(p)$ is either trivial or isomorphic to $\ZZ$. Conversely, if there is a point $p$ such that  the group $\Stab_{G_{\alpha}}(p)$ is neither $\ZZ$ nor trivial, then $\alpha$ is a relation number.
\end{prop}
\begin{proof}
It is  well known that in our setting, $\Stab_{G_{\alpha}}(p)$  is solvable. When $\alpha$ is a free number,   
$\Stab_{G_{\alpha}}(p)$  is trivial or isomorphic to $\ZZ$ since $G_\alpha$ is a free group of rank $2.$ The second part is the contrapositive of the first part. 
\end{proof}

 Since $A \in \Stab_{G_{\alpha}}(\infty)$ and $B_{\alpha} \in \Stab_{G_{\alpha}}(0)$, Corollary \ref{pointstabilizer} follows.
Before proving this, we need to define the length of $g \in G_{\alpha}.$ This length depends on its word representation. Let $\alpha$ be a real positive number and  $F_2(x_1,x_2)$  the free group of rank $2$  with a free basis $\{ x_1 , x_2 \}.$ 
For a reduced word $w=x_{\sigma_{1}} ^{p_1} \cdots x_{\sigma_{k}} ^{p_k}$ of $F_2(x_1,x_2)$, we define the $\textsf{length}$ of $w$ to be $k$, where $\sigma_{i}\in\{1,2\}$ for all $i\in \{1, 2, \cdots, k\}.$ 
Now we consider a homomorphism $q_{\alpha}:F_2(x_1,x_2) \to G_{\alpha}$ defined by $x_1 \mapsto A, x_2 \mapsto B_{\alpha}.$ 
For $g \in G_{\alpha}$, we say that $w$ is a $\textsf{lifting word}$ of $g$ if $q_{\alpha}(w)=g$ and $w$ is a reduced word in $F_2(x_1,x_2).$ 
If $G_{\alpha}$ is free, each element of $G_\alpha$ has a unique lifting word whereas if $\alpha$ is a relation number, then it does not. 
Nevertheless, we can consider the length of each lifting. Hence whenever we mention the length of $g \in G_{\alpha}$, it refers to the length of a particular lifting word of $g$.
% it always means that we fix a lifting word and it is the length of the word.

\begin{cor} \label{pointstabilizer}
 Let $\alpha$ be a positive real number. If $\alpha$ is a free number, then $$\Stab_{G_{\alpha}}(\infty)=\{ A^k : k \in \ZZ \} \textnormal{ and } \Stab_{G_{\alpha}}(0)=\{ B_{\alpha} ^k : k \in \ZZ \}.$$
\end{cor}
\begin{proof}
It follows at once that $\Stab_{G_{\alpha}}(\infty) \supset \{ A^k : k \in \ZZ \}$ so it suffices to show that $\Stab_{G_{\alpha}}(\infty) \subset \{ A^k : k \in \ZZ \}$. Assume that there is an element  $x$  in $\Stab_{G_{\alpha}}(\infty) - \{ A^k : k \in \ZZ \}.$  
Then since $B_{\alpha} ^k \not \in \Stab_{G_{\alpha}}(\infty)$, the length of $x$ is at least $2$. By Proposition \ref{prop4}, $\Stab_{G_{\alpha}}(\infty) \cong \ZZ$ and there is a generator $y$ of $\Stab_{G_{\alpha}}(\infty)$. 
Then $y^m = A$  and $y^n = x$ for some $m,n \in \ZZ.$ This gives that  $A^n = x^m.$ This gives a relation in $G_{\alpha}$ since the length of $x$ is at least 2 . 
This contradicts the freeness of $G_{\alpha}$. A similar argument can apply to the second one. Thus we are done. 
\end{proof}

%%%%%%%%%%%%%%%%%%%%%%%%%%%%%%%%%%%%%%%%%%%%%%%%%%%%%%%%%%%%%%%%%%%%%%

\section{The Generalized Farey Graph} \label{TheGeneralizedFareyGraph}
In this section, we define a graph $\Gamma_{\alpha}$ which is a generalization of the Farey graph, and prove Theorem \ref{thethm}.

\begin{defn}
 Let $\alpha$ be a real number. The \textsf{generalized Farey graph} $\Gamma_{\alpha}$ at $\alpha$ is the graph with vertex set 
$$V = \{ g(0) , g(\infty) : g \in G_{\alpha} \}$$ and edge set
$$E = \{ g \cdot \ell : g \in G_{\alpha} \},$$ where $\ell$ is the set $\{ 0 , \infty \}.$
\end{defn}

% To distinguish between vertices and edges clearly, we write a vertex $g(0)$ and $g(\infty)$ instead of $g \cdot 0$ and $g \cdot \infty$, repsectively. 

 The graph $\Gamma_{\alpha}$ has the following combinatorial properties.

\begin{lem}
Let $\alpha$ be a positive real number. Then the graph $\Gamma_{\alpha}$ is connected and not locally finite.
\end{lem}
\begin{proof}
Let $S$ be a subgraph of $\Gamma_\alpha$ with  vertex set 
$$V_S=\ell \cup (A \cdot \ell ) \cup (B_{\alpha} \cdot \ell) \cup (A^{-1} \cdot \ell) \cup (B_{\alpha} ^{-1} \cdot \ell)$$
and edge set 
$$E_S=\{\ell, A \cdot \ell,B_{\alpha} \cdot \ell , A^{-1} \cdot \ell, B_{\alpha} ^{-1} \cdot \ell \}.$$
Note that $S$ is a connected subgraph of $\Gamma_\alpha$ and observe that for any word $X\in F_2(x_1,x_2),$ the subgraph 
with  vertex set $q_\alpha(X) \cdot V_S$ and edge set $q_\alpha(X)\cdot E_S$ is connected. By induction on the word length of $X$, the graph $\Gamma_{\alpha}$ is connected. 
Moreover, $\Gamma_{\alpha}$ is not locally finite because the vertex $\infty$ is connected to all of integers by the edges \{ $A^n \cdot \ell : n\in \ZZ\}$.
\end{proof}
So the graph $\Gamma_{\alpha}$ shares some properties with the Farey graph. 
Indeed, the generalized Farey graph $\Gamma_{1}$ at $1$ is the Farey graph. It follows from the fact that $G_1=\SLZ.$ 
This is why we call the graph $\Gamma_{\alpha}$ the generalized Farey graph.

In Theorem \ref{thethm}, we use the fact that each relation of $G_{\alpha}$ corresponds to a cycle in $\Gamma_{\alpha}$. More precisely, for each reduced word $w$ of $F_2(x_1,x_2)$, there is a corresponding path in $\Gamma_{\alpha}$. 
The following lemma says that there is such a canonical correspondance. For convenience, a word $W=C_1 ^{p_1} \cdots C_k ^{p_k}$ in $G_{\alpha}$ means that $W\in G_{\alpha}$, $\{C_i,C_{i+1}\}= \{A, B_{\alpha}\}$ and there is a lifting word $w=x_{\sigma_1} ^{p_1} \cdots x_{\sigma_k} ^{p_k}$ of $W.$

\begin{lem} \label{wordtopath}
Let $\alpha$ be a positive real number and $W=C_1 ^{p_1} \cdots C_k ^{p_k}$ be a word of length $k>0$ in $G_{\alpha}$. Then a sequence of edges
$$ \ell , (C_1 ^{p_1} \cdot \ell) , (C_1 ^{p_1} C_2 ^{p_2} \cdot \ell),  \cdots , (C_1 ^{p_1} \cdots C_k ^{p_k} \cdot \ell) =W \cdot \ell $$ is a path without backtracking. Thus the length of this path is $k+1$.
\end{lem}
\begin{proof}
First note that  for any $p\in \ZZ-\{0\}$, $\ell \cap (A^p \cdot \ell) =\{\infty\}$ and $\ell \cap (B_\alpha^p\cdot\ell) =\{0\}.$ Hence for any $g\in G_\alpha$, $g\cdot \ell \cap (gA^p \cdot \ell) =\{g(\infty)\}$ and $g\cdot \ell \cap gB_\alpha^p \cdot  \ell =\{g(0)\}.$ This also implies that for any two integers $p$ and  $q$ in $\ZZ-\{0\},$  two finite sequences   
$$\ell , (A^p \cdot \ell) , (A^pB_\alpha^q \cdot \ell)$$
and 
$$\ell , (B_\alpha ^p \cdot \ell) , (B_\alpha^p A^q \cdot \ell)$$
are paths without backtracking . Therefore for any $g\in G_\alpha,$  two sequences $$g\cdot \ell , (gA^p \cdot \ell) , (gA^pB_\alpha^q \cdot \ell)$$
and 
$$g\cdot \ell , (gB_\alpha ^p \cdot \ell) , (gB_\alpha^p A^q \cdot \ell)$$
are paths without backtracking. See Figure \ref{figureedges}.

When $k$ equals $1,$ the result follows directly. Also when $k$ is greater than $1,$  the finite sequence $$ \ell , (C_1 ^{p_1} \cdot \ell) , (C_1 ^{p_1} C_2 ^{p_2} \cdot \ell),  \cdots , (C_1 ^{p_1} \cdots C_k ^{p_k} \cdot \ell) $$ is a path without backtracking. 
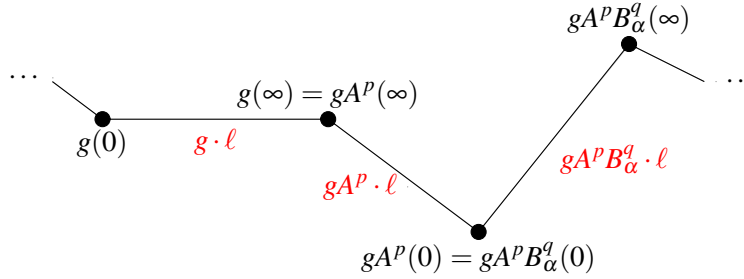
\begin{figure} [h!] 
\begin{center}
	\begin{tikzpicture}
	\draw [black,fill] (-2/3,1/2) -- (0,0) circle [radius=0.1] node [black,below] {$g(0)$} --  (3,0) circle [radius=0.1] node [black,above] {$g (\infty) = g A^{p} (\infty)$} -- (5,-1.5) circle [radius=0.1] node [black,below] {$g A^{p} ( 0) = g A^{p} B_{\alpha} ^{q} ( 0) $} -- (7,1)  circle [radius=0.1] node [black,above] {$g A^{p} B_{\alpha} ^{q} ( \infty) $} -- (8,0.5) ;
	\draw [black,fill] (1.5,-0.3) circle [radius=0] node [red] {$g \cdot \ell$};
	\draw [black,fill] (4.05,-0.9) circle [radius=0] node [red,left] {$g A^{p} \cdot \ell$};
	\draw [black,fill] (6.8,-0.53) circle [radius=0] node [red] {$g A^{p} B_{\alpha} ^{q} \cdot \ell$};
	\draw [black,fill] (-2/3,0.55) circle [radius=0] node [black,left] {$\cdots$};
	\draw [black,fill] (8.15,1/2) circle [radius=0] node [black,right] {$\cdots$};
	\end{tikzpicture}
\caption{The edges $g \cdot \ell , gA^{p} \cdot \ell$ and $g A^{p} B_{\alpha} ^{q} \cdot \ell$.}
\label{figureedges}
\end{center}
\end{figure}

\end{proof}

Then we discuss the elements in $\Stab_{G_\alpha}(\ell).$ The following lemma will imply one direction of Theorem \ref{thethm}.

\begin{lem} \label{edgestabilizer}
Let $\alpha$ be a positive real number and $g$ be an element in $\Stab_{G_\alpha}(\ell).$ Assume that there is a lifting word $s$ of $g$ in $F_2(x_1,x_2)$ such that $s$ is not the empty word. Then the length of $s$ is at least $3$.
\end{lem}
\begin{proof}
Since $s$ is not the empty word, we can write $g=q_\alpha(s)=C_1 ^{p_1} \cdots C_k ^{p_k}$ as a word in $G_\alpha$ for some $k\in \NN$. 
Note that for any $p\in \ZZ-\{0\}$, $\ell \cap (A^p \cdot \ell) =\{\infty\}$ and $\ell \cap (B_\alpha^p\cdot\ell) =\{0\}.$ 
Hence $k$ can not be $1.$ Now we assume that $k=2,$ that is   $q_\alpha(s)=C_1 ^{p_1} C_2 ^{p_2}.$

Since $g(\ell)=\ell,$ there are two possible cases: $\Fix(g)=\ell$; $\Fix(g)=\emptyset.$  
If $\Fix(g)=\ell$, then $g$ is of the form $$\begin{bmatrix}
a & 0 \\ 0 & \frac{1}{a}
\end{bmatrix}$$ for some $a\in \RR.$ 
On the other hand, 
 $$A^{p_1} B_{\alpha} ^{p_2} = \begin{bmatrix}
1 + p_1p_2 \alpha & p_1 \\ p_2 \alpha & 1 
\end{bmatrix} \textnormal{ and } B_{\alpha} ^{p_1} A^{p_2}=
\begin{bmatrix}
1 & p_2 \\ p_1 \alpha & 1+ p_1 p_2 \alpha
\end{bmatrix},$$
and this implies that $p_1$ and $p_2$ are zero. This is a contradiction. 

 Now we consider the case where $\Fix(g)=\emptyset.$ Then $g$ is of the form $$\begin{bmatrix}
0 & a \\ - \frac{1}{a} & 0
\end{bmatrix}$$ for some $a\in \RR.$ Similarly  this is also a contradiction. 
Thus the length of $s$ is at least $3$. 
\end{proof}

\begin{rmk}
 The above lemma does not hold for all edge $y$ in $\Gamma_{\alpha}$. For example, consider the case where $\alpha = 1$ and the edge is  $y=B_1 \cdot \ell$. 
Then $$M=A B_1 ^{-2} = \begin{bmatrix}
-1 & 1 \\ -2 & 1
\end{bmatrix}$$ fixes the edge $y$.
\end{rmk}

\begin{prop} \label{relationtonottree}
Let $\alpha$ be a positive real number. Assume that $\alpha$ is a relation number. Then the generalized Farey graph $\Gamma_{\alpha}$ is not tree.
\end{prop}
\begin{proof}
Since $\alpha$ is a relation number, there is a lifting word $w$ of the identity $I$ of $G_\alpha$ in $F_2(x_1,x_2)$ which is not the empty word. Then we  can write $W = q_\alpha(w)= C_1 ^{p_1} \cdots C_k ^{p_k}$ for some $k\in \NN$ as a word in $G_\alpha. $ Since $W=I\in \Stab_{G_\alpha}(\ell),$ by Lemma \ref{edgestabilizer}, we get that $k \geq 3$. 

Now we consider the corresponding sequence

$$\ell , (C_1 ^{p_1} \cdot \ell) , (C_1 ^{p_1} C_2 ^{p_2} \cdot \ell) , \cdots , (C_1 ^{p_1} \cdots C_k ^{p_k}\cdot \ell) =W \cdot \ell.$$
By Lemma \ref{wordtopath}, it is a path without backtracking and so the length of this path is at least $4$.
Therefore, this path has a cycle as a subgraph since $W\cdot \ell= \ell.$ Thus $\Gamma_{\alpha}$ is not tree.
\end{proof}

\begin{exam}
 Recall that $G_1 = \SLZ$ which is not a  free group of rank $2$. Then $(x_1^{-1}x_2)^6$ is one of the lifting words of the identity $I$. So the sequence of the edges
$$\ell , (A^{-1} \cdot \ell) , (A^{-1} B_1 \cdot \ell) , \cdots , (A^{-1}B_1)^6 \cdot \ell $$ has a cycle as a subgraph. Indeed, three edges $\ell , A^{-1} \cdot \ell , A^{-1}B \cdot \ell$ give a cycle. This implies that the graph $\Gamma_1$ is not tree.
\end{exam}

 Hence, Proposition \ref{relationtonottree} gives one part of Theorem \ref{thethm}.
To prove the converse, we will show that if $\alpha$ is a free number, then $\Gamma_{\alpha}$ must be tree.
\begin{lem} \label{freestabilizer}
Let $\alpha$ be a positive real number. If $\alpha$ is a free number, every edge stabilizer in the graph $\Gamma_{\alpha}$ is trivial.
\end{lem}
\begin{proof}
 Choose an element $g \in \Stab_{G_{\alpha}}(\ell)$. Then there are two cases: $\Fix(g)=\ell;$ $\Fix(g)=\emptyset.$
 First when $\Fix(g)=\emptyset,$
the element $g$ is of the form $$\begin{bmatrix}
0 & a \\ -\frac{1}{a} & 0
\end{bmatrix}$$ for some $a\in \RR.$ Then 
$$s^4=\begin{bmatrix}
0 & a \\ -\frac{1}{a} & 0
\end{bmatrix}^4 = I$$
and so $s$ is of order $4.$
However this is a contradiction since $G_{\alpha}$ is torsion free. Therefore $s$ must fix $0$ and $\infty.$
From Corollary \ref{pointstabilizer}, we have that $\Stab_{G_{\alpha}}(0) \cap \Stab_{G_{\alpha}}(\infty)$ is trivial. Thus we have $$\Stab_{G_{\alpha}}(\ell)=1.$$
Choose an arbitrary edge in $\Gamma_{\alpha}$, say $g \cdot \ell$. Then note that
$$\Stab_{G_{\alpha}}(g \cdot \ell) = \{ g x g^{-1} : x \in \Stab_{G_{\alpha}}(\ell)\}=g\Stab_{G_\alpha}(\ell)g^{-1}.$$
Hence  $\Stab_{G_{\alpha}}(g \cdot \ell)$ is trivial as $\Stab_{G_{\alpha}}(\ell)$ is trivial.
\end{proof}

\begin{lem} \label{vertexadjacent}
Let $\alpha$ be a positive real number. When $\alpha$ is a free relation number, the sets of edges whose endpoints are $\infty$ and $0$ are $$\{ A^k \cdot \ell : k \in \ZZ \} \textit{ and } \{ B_{\alpha} ^k \cdot \ell : k \in \ZZ \}, \textit{ respectively.}$$
\end{lem}
\begin{proof}
First suppose that  $g\cdot \ell$ is an edge adjacent to $\infty.$
Then either $g ( \infty)=\infty$ or $g(0)=\infty$.
If $g( \infty)=\infty$, then  $g \in \Stab_{G_{\alpha}}(\infty)$. By Corollary \ref{pointstabilizer}, $g=A^k$ for some $k \in \ZZ$. Then we consider the case where $g ( 0)=\infty.$ Let $M=g B_{\alpha} g^{-1}.$ 
Then $M \in \Stab_{G_{\alpha}}(\infty)$ and  by Corollary \ref{pointstabilizer}, we obtain that $g B_{\alpha} g^{-1} = A^k $ for some $k$. 
On the other hand, since $G_\alpha$ is a free group of rank $2,$ we can define a homomorphism $e_{\alpha,2}:G_{\alpha} \to \ZZ$ by $A \mapsto 0$ and $B_{\alpha} \mapsto 1.$ 
Since $e_{\alpha,2} (g B_{\alpha} g^{-1} A^{-k})=1 \neq 0,$ we can conclude that  $g B_{\alpha} g^{-1} A^{-k} \neq I.$ This is a contradiction. Thus, when $g \cdot \ell$ is an edge adjacent to $\infty$, $g$ must be in the set $\{ A^k : k \in \ZZ \}$. 
Similarly we can obtain the second statement.
\end{proof}

\begin{cor} \label{edgeadjacent}
Let $\alpha$ be a positive real number. Assume that $\alpha$ is a free number. 
Then for each edge $g \cdot \ell$ of $\Gamma_\alpha$, the sets of edges whose endpoints are $g(\infty)$ and $g(0)$ are $$ \{ g A^k \cdot \ell : k \in \ZZ  \} \ \text{and} \  \{ g B_{\alpha} ^k \cdot \ell : k \in \ZZ  \}, \ respectively,$$ 
and the set of edges adjacent to the edge $g \cdot \ell$ is  $$ \left \{ g A^k \cdot \ell : k \in \ZZ-\{0\}  \right \} \cup \left \{ g B_{\alpha} ^k \cdot \ell : k \in \ZZ-\{0\} \right \}.$$
\end{cor}
\begin{proof}
It follows immediately from Lemma \ref{vertexadjacent}.
\end{proof}

The following lemma enables us to compare the length of a word in $G_\alpha$ with the length of the corresponding path.
\begin{lem} \label{pathrule}
Let $\alpha$ be a positive real number and $\{D_i\}_{i=1}^k$ be a finite sequence in $G_\alpha$ for some $k\in \NN$ with $k>1$. Assume that $\alpha$ is a free number and the sequence $$(D_1 \cdot \ell) , (D_2 \cdot \ell) , \cdots , (D_k \cdot \ell)$$ is a path without backtracking in $\Gamma_{\alpha}$. 
Then $$ D_i ^{-1} D_{i+1} \in \left \{ A^p , B_{\alpha} ^q : p,q \in \ZZ-\{ 0 \} \right \}$$ for all $i\in\NN$ with $1 \leq i \leq k-1.$
Moreover, for each $i\in \NN$ with $1 \leq i \leq k-2$, if $ D_i ^{-1} D_{i+1} = A^{p_0}$ for some nonzero integer $p_0$, then $D_{i+1} ^{-1} D_{i+2} = B_{\alpha}^{q_0}$ for some nonzero integer $q_0$, and vice versa.
\end{lem}
\begin{proof}
The first part follows directly from Corollary \ref{edgeadjacent}. Then observe that for each two integers $n$ and $m$ in $\ZZ-\{0\},$  two finite sequences $$\ell, (A^n \cdot \ell ), (A^{n+m} \cdot \ell)$$ and  $$\ell, (B^n \cdot \ell ), (B^{n+m} \cdot \ell)$$ are not paths.
Hence   for each two integers $n$ and $m$ in $\ZZ-\{0\}$ and for each $g\in G_\alpha,$ two finite sequences $$g\cdot \ell, (g A^n \cdot \ell ), (gA^{n+m} \cdot \ell)$$ and  $$g\cdot \ell, (gB^n \cdot \ell ), (gB^{n+m} \cdot \ell)$$ are not paths. The second statement follows from this observation. 
\end{proof}

\begin{cor} \label{pathlengthtowordlength}
Let $\alpha$ be a positive real number and $\{g_i\}_{i=1}^k$ be a finite sequence in $G_\alpha$ for some $k\in \NN$ with $k>1.$ 
Suppose that $\alpha$ is a free number. Assume that the finite sequence $\{g_i \cdot \ell \}_{i=1}^k$ is a path of length $k$ beginning at $\ell$, that is $g_1=\id_{G_\alpha}.$ Then the length of the word $g_k$ in $G_{\alpha}$ is $k-1$.
\end{cor}
\begin{proof}
It is obtained from Lemma \ref{pathrule}.
\end{proof}

 Now let us prove the main theorem of this section.
 
\begin{thm} \label{thethm}
Let $\alpha$ be a positive real number. Then $\alpha$ is a relation number if and only if the generalized Farey graph $\Gamma_{\alpha}$ at $\alpha$ is not tree.
\end{thm}
\begin{proof}
One direction has been shown by Proposition \ref{relationtonottree}. It thus suffices to show that  if $\Gamma_{\alpha}$ is not tree, then $\alpha$ is a relation number. 
We claim that if $G_{\alpha}$ is a free group of rank $2$, then the graph $\Gamma_{\alpha}$ is tree. 
Suppose that $G_{\alpha}$ is a free group of rank $2$ and $\Gamma_{\alpha}$ is not tree.
Since $\Gamma_{\alpha}$ is not tree, there is a cycle and we may assume that the cycle contains $\ell.$ 
Hence there is a finite sequence $\{D_i\}_{i=0}^k$ in $G_\alpha$ for some $k\in \NN$ with $ k \geq 3$ such that $D_0=D_k=\id_{G_\alpha}$ and the sequence  $\{ D_i \cdot \ell \}_{i=0}^k$ is a cycle. Now we consider the path $$ \ell = D_0 \cdot \ell , \cdots , D_{k-1} \cdot \ell . $$ 
Then the length of this path is $k$ and  by Corollary \ref{pathlengthtowordlength}, the length of $D_{k-1}$ is $k-1$. 
On the other hand, since the edge $D_{k-1} \cdot \ell$ is adjacent to $ D_k \cdot \ell = \ell$, the length of the word $D_{k-1}$ is 1 by Lemma \ref{vertexadjacent}. Then since $G_{\alpha}$ is a free group of rank $2$, we have $k-1=1$. 
However, $k \geq 3$ so this is impossible. Therefore, the theorem is proved.
\end{proof}

%%%%%%%%%%%%%%%%%%%%%%%%%%%%%%%%%%%%%%%%%%%%%%%%%%%%%%%%%%%%%%%%%%%%%%

\section{Applications : The Orbit Test for Relation Numbers} \label{orbittestforrelaion}
In this section, we prove Proposition \ref{OrbitTest}, called the orbit test. 
Under the condition of Proposition \ref{OrbitTest}, we can find a cycle in the generalized Farey graph $\Gamma_\alpha$ and  by Theorem \ref{thethm}, we deduce that $\Gamma_\alpha$ is not tree. 
Therefore, we give some new examples of relation numbers which we find using Proposition \ref{OrbitTest}.

We first prove some lemmas. Let $\alpha$ be a positive real number. Then the first lemma says that the graph $\Gamma_{\alpha}$ is symmetric with respect to the geodesic $\ell$. 
To observe this, we define a map $h : G_\alpha \to G_\alpha$ as follows. For each element $M\in G_\alpha,$ we can write 
$$M=\begin{bmatrix}a&b\\ c&d\end{bmatrix}$$
and define $$h(M)= \begin{bmatrix}a&-b\\ -c&d\end{bmatrix}.$$ Then we can see that the map $h$ is an isomorphism on $G_\alpha$ and that  $h(A)=A^{-1}$ and $h(B_\alpha)=B_{\alpha} ^{-1}.$ 
Hence an automorphism $\tilde{h}$ of $F(x_1,x_2)$ defined by $x_1\mapsto x_1^{-1}$ and $x_2\mapsto x_2^{-1}$ makes the following diagram commute. 
\begin{center}
\begin{tikzcd}
F_2(x_1, x_2)\arrow[r, "\tilde{h}"]\arrow[swap, d, "q_\alpha"]& F_2(x_1,x_2)\arrow[d, "q_\alpha"]\\
G_\alpha \arrow[swap]{r}{h} &G_\alpha
\end{tikzcd}
\end{center}
The following lemma follows from the definition of the map $h.$

\begin{lem} \label{GFGsymmetric}
Let $\alpha$ be a positive real number and  $X$ be an element in $G_{\alpha}.$ If $X(0) \in \RR$, then $h(X)( 0)=-X( 0)$. Similarly, if $X (\infty) \in \RR$, then $h(X) (\infty) = -X ( \infty).$
\end{lem}
Lemma \ref{GFGsymmetric} tells us that the generalized Farey graph $\Gamma_{\alpha}$ is symmetric with respect to $\ell$.

\begin{lem} \label{1/2}
Let $\alpha$ be a positive real number and $x$ be an element in $F(x_1,x_2).$ Assume that $q_\alpha(x)(0)=1/2$ (or $q_\alpha(x) ( \infty) = 1/2).$ 
Then an element $y=x_1\tilde{h}(x)$ in $F_2(x_1,x_2)$ satisfies that $y\neq x$ and  $q_\alpha(y)( 0 )= 1/2$ (or $q_\alpha(y)(\infty) = 1/2$, respectively).
\end{lem}
\begin{proof}
First we consider the case where $q_\alpha(x) (0) = 1/2$. By Lemma \ref{GFGsymmetric}, we have that $h(q_\alpha(x))( 0) = - 1/2$. 
Therefore $$q_\alpha(y)(0)=q_\alpha(x_1\tilde{h}(x))(0)=q_\alpha(x_1)q_\alpha(\tilde{h}(x))(0)= A h(q_\alpha(x))( 0)= A ( - 1/2)=1/2.$$ 

Then we show that $y\neq x.$ Consider  a homomorphism $e:F_2(x_1,x_2)\to \ZZ$ by $x_1 \mapsto 1$ and $x_2 \mapsto 0.$ Since $e(x^{-1}y)=e(x^{-1}x_1\tilde{h}(x)) \neq 0,$ then $x^{-1}y\neq 1.$ 
Likewise we can show the case where  $q_\alpha(x) ( \infty) = 1/2.$ 
\end{proof}

The last ingredient is quite combinatorial. This lemma follows from Lemma \ref{DefTree} so we just state the lemma without proof.

\begin{lem} \label{TreeLem}
 Let $X$ be a connected graph. Choose a vertex $v$ in $X$. Suppose that there exist two distinct edges $e_1$ and  $e_2$ having $v$ as endpoints. Suppose that  two paths $P_1$ and  $P_2$ in $X$ are given with the following properties:
\begin{itemize}
 	\item the first edges of $P_1$ and $P_2$ are $e_1$ and $e_2$, respectively;
	\item two paths $P_1$ and $P_2$ have no backtrackings;
	\item two paths $P_1$ and $P_2$ have  the same starting points and the same terminal points. 
	
\end{itemize}
Then $X$ is not tree.
\end{lem}

 Now we give a proof of the orbit test.

\begin{prop}[The orbit test] \label{OrbitTest}
 Let $\alpha$ be a real number. Suppose that the $G_{\alpha}$- orbit of $0$ or $\infty$ contains the number $1/2$, that is, for some $X \in G_{\alpha}$, $X(0) = 1/2$ or $X (\infty) = 1/2$. Then $\alpha$ is a relation number.
\end{prop}
\begin{proof}
Recall that there is no such  $X$ in $G_0.$ Hence we just consider the case where $\alpha$ is not $0.$  As mentioned previously, we may also assume that $\alpha$ is a positive real number. 
Then suppose that there is an element $X$ in $G_{\alpha}$ such that $X(0) = 1/2.$ Let $x$ be a lifting word of $X.$ Since $X(0) = 1/2$, the element $x$ is not the empty word. 
Hence we can write $x=x_{\sigma_1}^{p_1}\cdots x_{\sigma_i}^{p_i}$ for some $i\in \NN$ in $F_2(x_1,x_2).$ Then by Lemma  \ref{1/2}, the element $y=x_1\tilde{h}(x)$ differs with $x$ and $q_\alpha(y)(0)=1/2.$ 
Since $y$ is also not the empty word, we can write $y=x_{\delta_1}^{q_1}\cdots x_{\delta_j}^{q_j}$ for some $j\in \NN$ so that the word  $x_{\delta_1}^{q_1}\cdots x_{\delta_j}^{q_j}$ is the reduced word of $y$ in $F_2(x_1,x_2).$ 
Then we define two finite sequences $P_x$ and $P_y$ of edges in $\Gamma_\alpha$ as follows. The sequence $P_x$ is 
$$\ell, (q_\alpha(x_{\sigma_1}^{p_1})\cdot \ell), (q_\alpha(x_{\sigma_1}^{p_1} x_{\sigma_2}^{p_2})\cdot \ell), \cdots, (q_\alpha(x_{\sigma_1}^{p_1}x_{\sigma_2}^{p_2}\cdots x_{\sigma_i}^{p_i})\cdot \ell)=(X\cdot \ell)$$ 
and the sequence $P_y$ is  
$$\ell, (q_\alpha(x_{\delta_1}^{q_1})\cdot \ell), (q_\alpha(x_{\delta_1}^{q_1} x_{\delta_2}^{q_2})\cdot \ell), \cdots, (q_\alpha(x_{\delta_1}^{q_1}x_{\delta_2}^{q_2}\cdots x_{\delta_j}^{q_j})\cdot \ell)=(q_\alpha(y)\cdot \ell).$$ 
The  sequences $P_x$ and $P_y$ are paths without backtracking by Lemma \ref{wordtopath}. 

Now we claim that there are two paths in $\Gamma_\alpha$ satisfying the condition of \ref{TreeLem}. First we consider the case where $x_{\sigma_1}=x_1.$ Since $X(0) = 1/2$, the number $i$ is at least $2$ and the length of $P_x$ is at least $3.$ 
If $p_1=1$, then $x_{\delta_1}=x_2$ by the choice of $y$ and the length of $P_y$ is at least $2.$ In this case,  the starting point of $P_x$ is $0$ and the starting point of $P_y$ is $\infty.$ 
Hence we consider two paths $Q_x$ and $Q_y$ in $\Gamma_\alpha$ such that $Q_y$ is $P_y$ and $Q_x$ is 
$$ (q_\alpha(x_{\sigma_1}^{p_1})\cdot \ell), (q_\alpha(x_{\sigma_1}^{p_1} x_{\sigma_2}^{p_2})\cdot \ell), \cdots, (q_\alpha(x_{\sigma_1}^{p_1}x_{\sigma_2}^{p_2}\cdots x_{\sigma_i}^{p_i})\cdot \ell).$$
Then the paths $Q_x$ and $Q_y$ without backtracking have $\infty$ as the starting points and their lengths are at least $2.$ Moreover the first edges of $Q_x$ and $Q_y$ differ. 
Then since $X(0)=1/2$ and $q_\alpha(y)(0)=1/2,$ both $Q_x$ and $Q_y$ have $1/2$ as vertices. 
Hence we can take two paths $R_x$ and $R_y$ so that $R_x$ and $R_y$ are induced subgraphs of $Q_x$ and $Q_y$, respectively, both starting at $\infty$ and ending at $1/2.$ 
Two paths $R_x$ and $R_y$ satisfy the condition of Lemma \ref{TreeLem}. 

Then we consider the case where $p_1\neq 1.$ In this case, $x_{\sigma_1}=x_{\delta_2}=x_1$, $p_1\neq q_1$ and also the lengths of $P_x$ and $P_y$ are equal by the choice of $y.$ 
Hence the starting points of $P_x$ and $P_y$ are $0,$ and the first edges of $P_x$ and $P_y$ are the same whereas the second edges of $P_x$ and $P_y$ are different. Now we define two paths $Q_x$ and $Q_y$ as follows.
The path $Q_x$ is 
$$(q_\alpha(x_{\sigma_1}^{p_1})\cdot \ell), (q_\alpha(x_{\sigma_1}^{p_1} x_{\sigma_2}^{p_2})\cdot \ell), \cdots, (q_\alpha(x_{\sigma_1}^{p_1}x_{\sigma_2}^{p_2}\cdots x_{\sigma_i}^{p_i})\cdot \ell)$$ 
and the path $Q_y$ is 
$$ (q_\alpha(x_{\delta_1}^{q_1})\cdot \ell), (q_\alpha(x_{\delta_1}^{q_1} x_{\delta_2}^{q_2})\cdot \ell), \cdots, (q_\alpha(x_{\delta_1}^{q_1}x_{\delta_2}^{q_2}\cdots x_{\delta_j}^{q_j})\cdot \ell).$$ 
Then $Q_x$ and $Q_y$ are induced subgraphs of $P_x$ and $P_y$, respectively, and the first edges of $Q_x$ and $Q_y$ differ and the starting points of $Q_x$ and $Q_y$ are $\infty.$ 
Since $X(0)=1/2$ and $q_\alpha(y)(0)=1/2,$ both $Q_x$ and $Q_y$ have $1/2$ as vertices. Hence we can take two paths $R_x$ and $R_y$ so that $R_x$ and $R_y$ are induced subgraphs of $Q_x$ and $Q_y$, respectively, both starting at  $\infty$ and ending at  $1/2.$ 
The paths $R_x$ and $R_y$ satisfy the condition of Lemma \ref{TreeLem}. 

Finally we consider the case where $x_{\sigma_1}=x_2.$ Then the length of $P_x$ is at least $3,$  $x_{\delta_1}=x_1$ and the length of $P_y$ is at least $4.$ Hence the starting point of $P_x$ is $\infty$ and the starting point of $P_y$ is $0.$  Now we define two paths $Q_x$ and $Q_y$ as follows.
The path $Q_x$ is the finite sequence 
$$(q_\alpha(x_{\sigma_1}^{p_1})\cdot \ell), (q_\alpha(x_{\sigma_1}^{p_1} x_{\sigma_2}^{p_2})\cdot \ell), \cdots, (q_\alpha(x_{\sigma_1}^{p_1}x_{\sigma_2}^{p_2}\cdots x_{\sigma_i}^{p_i})\cdot \ell)$$ 
and the path $Q_y$ is $P_y.$ Then $Q_x$ and $Q_y$ are induced subgraphs of $P_x$ and $P_y$, respectively, the first edges of $Q_x$ and $Q_y$ differ and the starting points of $Q_x$ and $Q_y$ are $0.$ Since $X(0)=1/2$ and $q_\alpha(y)(0)=1/2,$ both $Q_x$ and $Q_y$ have $1/2$ as vertices. 
Hence we can take two paths $R_x$ and $R_y$ so that $R_x$ and $R_y$ are induced subgraphs of $Q_x$ and $Q_y$, respectively,  both starting at $0$ and ending at $1/2.$ 
The paths $R_x$ and $R_y$ satisfy the condition of Lemma \ref{TreeLem}. Thus the claim is proved and by Lemma \ref{TreeLem}, $\Gamma_\alpha$ is not tree. 
Likewise we can prove the case where there is an element $X$ in $G_{\alpha}$ such that $X(\infty) = 1/2.$
\end{proof}

\begin{rmk}
Suppose that there is an element $X$ in $G_\alpha$ such that $X(0) = (2n+1)/2$ for some integer $n \in \ZZ.$ Then $A^{-n}X(0) = 1/2$ so we have that $\alpha$ is a relation number. Hence Proposition \ref{OrbitTest} still holds if we replace $1/2$ by $(2n+1)/2$ for any $n \in \ZZ.$
\end{rmk}

\begin{exam}
 Before we give a list of new rational relation numbers, we explain our orbit test with a specific example. Let $\alpha = 2/3$. Then $B_{\alpha}(\infty) = 3/2$ so $A^{-1} B_{\alpha}(\infty)=1/2.$ Then it follows directly that $A B_{\alpha}^{-1}(\infty)=-1/2$ and $A^2 B_{\alpha}^{-1}(\infty)=1/2.$ Consider two paths $$P_1 = A^{-1} \cdot \ell , A^{-1} B_{\alpha} \cdot \ell \textnormal{ and } P_2 = A^2 \cdot \ell , A^2 B_{\alpha} ^{-1} \cdot \ell.$$ Two paths $P_1$ and $P_2$ are different paths from $\infty$ to $1/2$. So the graph $\Gamma_{2/3}$ is not tree. Hence the rational $2/3$ is a relation number.
\end{exam}

\begin{rmk}
One of the simplest example of relation number is  $1/n$ for nonzero integer $n$ because $B_{\alpha} ^{2n}( \infty)=1/2$. 
 Moreover, we can see that $2/n$ and $3/n$ are  relation numbers by checking $B_{\alpha} (\infty) = n/2$ and $B_{\alpha} ^n A^{-1}( 0) = 1/2$, respectively.
One of non-trivial and so far unknown relation numbers is $41/18$. 
Note that $$B_{\alpha} ^{-2} A B_{\alpha} ^{-1} A (0) = \frac{45}{2}.$$ This gives the non-freeness of $G_{41/18}$. The following table is the list of the new examples of relation numbers with denominators $ \leq 30.$

\begin{table}[h!]
\begin{center}
	\label{tab:table1}
\begin{tabular}{|c|c||c|c|}
	\hline
	\textbf{$\alpha$} & \textbf{Orbit value} & \textbf{$\alpha$} & \textbf{Orbit value}\\
	\hline
	$\frac{41}{18}$ & $ B^{-2}AB^{-1}A \cdot 0 = \frac{45}{2} $ &
	$\frac{33}{19}$ & $ B^{-3}AB^{-1}A^2 \cdot 0 = \frac{171}{2} $ \\
	\hline
	$\frac{31}{20}$ & $ B^{-2}A^2B^{-1}A^2B^{-1}A^2 \cdot 0 = \frac{105}{2} $ &
	$\frac{35}{22}$ & $ B^{-2}A^{-1}BA^{-2}BA^{-3} \cdot 0 = \frac{2893}{2} $ \\
	\hline
	$\frac{43}{24}$ & $ B^2 A^4 B^{-1} A B^{-1} A^{-1} \cdot 0 = \frac{51}{2} $ &
	$\frac{41}{25}$ & $ B^5 A^{-1} B A^{-2} \cdot 0 = \frac{35}{2}$ \\
	\hline
	$\frac{57}{25}$ & $ B^{-2} A B^{-1} A \cdot 0 = \frac{175}{2} $ &
	$\frac{33}{26}$ & $B A^{-1} B^4 A^{-3} \cdot 0 = \frac{949}{2}$ \\
	\hline
	$\frac{35}{26}$ & $ B A^{-2} B A^{-3} B A^{-3} \cdot 0 = \frac{1001}{2}$ &
	$\frac{59}{26}$ & $ B^{-1} A B^{-1} A^2 \cdot 0 = \frac{65}{2}$ \\
	\hline
	$\frac{35}{27}$ & $B^{-2}A^{-3}BA^{-1} \cdot 0 = \frac{27}{2}$ &
	$\frac{43}{27}$ & $B^{-2}A^2 B^{-1}A \cdot 0 =  \frac{135}{2}$ \\
	\hline
	$\frac{33}{28}$ & $B A^{-2} B A^{-4} \cdot 0 = \frac{21}{2}$ &
	$\frac{37}{30}$ & $B^{-1} A B^{-4} A^3 \cdot 0 = \frac{45}{2}$ \\
	\hline
   \end{tabular}
	\caption{New rational relation numbers ($\alpha$ is omitted.)}
\end{center}
\end{table}
\noindent Recall that the orbit test can apply when $\alpha$ is an irrational number. By using the orbit test, we can also prove that $\alpha=2+\sqrt{2}$ is a relation number since $B_{\alpha} A^{-1} B_{\alpha}( \infty) = 1/2.$
\end{rmk}

 We do not know whether the converse of the orbit test holds. We expect it to be  false, but we could not find any counterexample.

%%%%%%%%%%%%%%%%%%%%%%%%%%%%%%%%%%%%%%%%%%%%%%%%%%%%%%%%%%%%%%%%%%%%%%

\section{Proof of the Main Theorem} \label{proofofmainthm}
In this section, we prove our main theorem, Theorem \ref{themainthm}.
Before proving some lemmas and the main theorem, we need to introduce some notations.
From now on, we regard $\alpha$ as a complex variable. Then  for each $n\in \NN,$ we write
 $$ \left( B_{\alpha} A^{-1} \right) ^n = \left( \begin{bmatrix}
1 & -1 \\ \alpha & 1 - \alpha
\end{bmatrix} \right) ^n = \begin{bmatrix}
t_n(\alpha) & u_n(\alpha) \\ m_n(\alpha) & l_n(\alpha)
\end{bmatrix}.$$  We get that for each $n\in \NN,$ $(B_{\alpha} A^{-1})^n (0) = u_n(\alpha) / l_n(\alpha).$ Moreover  for each $n\in \NN,$
$$ \begin{bmatrix}
u_{n+1} (\alpha) \\ l_{n+1}(\alpha)
\end{bmatrix}=\begin{bmatrix}
1 & -1 \\ \alpha & 1-\alpha
\end{bmatrix} \begin{bmatrix}
u_n(\alpha) \\ l_n (\alpha)
\end{bmatrix}$$ with $u_1(\alpha)=-1$ and $l_1(\alpha)=1-\alpha$. 
It is not difficult to show that for each $n\in \NN,$ the polynomial $u_n$ of the complex variable $\alpha$ is of degree $n-1$ and the polynomial $l_n$ is of degree $n.$
 Both $u_n$ and $l_n$ have integer coefficients and their leading coefficients are $(-1)^n$. 
Note that by Proposition \ref{OrbitTest}, a real number $\alpha_0$ is a relation number when $u_n(\alpha_0) / l_n(\alpha_0) = 1/2.$
In other words, for each $n\in \NN,$ we define a polynomial $$p_n(\alpha) = (-1)^{n+1} \left( 2 u_n(\alpha) - l_n(\alpha) \right),$$ and  then all real roots of $p_n(\alpha)$ are relation numbers. 
Note that $p_n(\alpha)$ is a monic polynomial with integer coefficients. 
We will show that the polynomials $p_n(\alpha)$ satisfy the conditions of the following theorem.

\begin{thm} \label{themainthm}
 There exists a sequence of polynomials $p_n(\alpha)$ satisfying the following:
\begin{itemize}
	\item each polynomial $p_n(\alpha)$ is a monic polynomial of degree $n$ with integer coefficients;
	\item all roots of $p_n(\alpha)$ are distinct real numbers and are relation numbers; and
	\item let $\alpha_n$ be the maximal root of $p_n(\alpha)$. Then the sequence $\{\alpha_n\}_{n=1}^\infty$ is increasing and converges to $4$.
\end{itemize}
\end{thm}

 We have shown the first two conditions except that all roots are real.
Thus the only remaining part is to prove that all roots are real and to verify the last condition.

\begin{lem} \label{rotationlemma}
Let $n$  be a natural number. The map $c_n:\RR \to \RR \cup \{ \infty \} = S^1$ defined by $\alpha \mapsto (B_{\alpha} A^{-1})^n (0)$ is a clockwise map.
\end{lem}
\begin{proof} Recall that $(B_{\alpha} A^{-1})^n(0) = u_n(\alpha)/l_n(\alpha).$ 
 We use induction on $n$. First we consider the case where $n=1.$ Note that $$c_1(\alpha)=u_1(\alpha)/l_1(\alpha) = \frac{1}{\alpha-1}.$$
Then the map $c_1$ is strictly decreasing at every point in $\RR-\{1\}$. At $\alpha=1$, $$-\frac{1}{c_1(\alpha)}=1-\alpha$$ and this is strictly decreasing. Hence the map $c_1$ is a clockwise map. Now we  assume that $c_n$ is a clockwise map. The map $c_{n+1}$ can be expressed as 
\begin{align}
    \label{eqn:inductivestep}
    c_{n+1}(\alpha) = \frac{u_{n+1}(\alpha)}{l_{n+1}(\alpha)}=\frac{u_n(\alpha)-l_n(\alpha)}{\alpha u_n(\alpha)+ (1-\alpha)l_n(\alpha)} = \frac{1}{\alpha + \frac{l_n(\alpha)}{u_n(\alpha) - l_n(\alpha)}} = \frac{1}{\alpha + \frac{1}{\frac{u_n(\alpha)}{l_n(\alpha)}-1}}= \frac{1}{\alpha + \frac{1}{c_n(\alpha)-1}}.
  \end{align}

 By Lemma \ref{rotationplusconstant}, we see that the map $$c_n(\alpha)-1$$ is clockwise
and by Lemma \ref{rotationinverse}, the map $$\frac{1}{c_n(\alpha)-1}$$ is anticlockwise.
 Lemma \ref{rotationplusx} then shows that $$\frac{1}{c_n(\alpha)-1} + \alpha$$ is also anticlockwise. 
Finally,  by Lemma \ref{rotationinverse}, the map $$\frac{1}{\alpha+\frac{1}{c_n(\alpha)-1}}$$ is clockwise and thus the map $c_{n+1}$ is also a clockwise map. 
\end{proof}

 Observe in  the equation \eqref{eqn:inductivestep} that  for all $n\in \NN,$ $\displaystyle \lim_{\alpha \to -\infty} c_n(\alpha) = \lim_{\alpha \to \infty} c_n(\alpha) = 0.$  Then, as discussed in Section \ref{winding}, we can consider the extension map $\overline{c_n}$ of $c_n.$

\begin{lem} \label{windingnumber}
 For each $n\in \NN,$ $$w \left( \overline{c_n} \right)=n.$$
\end{lem}

\begin{proof}
 We use mathematical induction on $n$. First we consider the case where $n=1.$ Then  $$c_1(\alpha) = \frac{1}{\alpha-1}$$ so the map $\overline{c_1}$ is a homeomorphism of $S^1.$ Therefore, $w(\overline{c_1})=1$ by Lemma  \ref{winding-calculuation}.
 Now we assume  that $w(\overline{c_n})=n$. By Lemma \ref{windinginverse}, we have that $$w(\overline{c_{n+1}}) = w\left(\overline{\frac{1}{c_{n+1}}}\right).$$ For brevity, we put $$ f(\alpha)=\frac{1}{c_{n+1}(\alpha)}.$$ 
Note that the extension map $\overline{f}$ is a covering map from $S^1$ to $S^1$ and  $$f(\alpha)= \alpha + \frac{1}{c_n(\alpha)-1}$$ from the equation \eqref{eqn:inductivestep}. 
Then we claim that $w(\overline{f})=n+1.$ Since $\overline{c_n}$ is a $n$-fold covering,  we obtain that $$|\{ \alpha \in S^1 :\overline{ c_n}(\alpha)=1 \}|=n.$$
Now we write $\{\alpha \in S^1 : \overline{c_n}(\alpha)=1 \}=\{x_1,x_2, \cdots, x_n\}.$ Note that $\{x_1,x_2, \cdots, x_n\} \subset \RR$ as $\overline{c_n}(\infty)=0.$
Hence $\overline{f}(x_i)=\infty$ for all $i\in \{1,2, \cdots, n\}.$ 
Moreover $\overline{f}(\infty)=\infty$ as $\displaystyle \lim_{\alpha \to -\infty} c_n(\alpha) = \lim_{\alpha \to \infty} c_n(\alpha) = 0.$ 
Therefore these imply that   $|\{x_1, x_2, \cdots , x_n, \infty \}|=|\overline{f}^{-1}(\{\infty\})|=n+1$. Thus $w(\overline{f})=|\overline{f}^{-1}(\{\infty\})|=n+1$ by Lemma  \ref{winding-calculuation} so we are done by induction. 
\end{proof}

\begin{prop} \label{rootreal}
For each $n\in \NN,$ all roots of the polynomial $p_n(\alpha)$ are distinct real numbers.
\end{prop}
\begin{proof}
First we claim that for each $n\in \NN$, the set $$\{ \alpha_0 \in \CC : u_n(\alpha_0) = 0 , \ l_n(\alpha_0) = 0\}$$ is empty. When $n=1$, the statement is trivial since $u_1(\alpha) = -1.$ Now we assume that the statement holds for $n$. Choose a number  $x\in \RR$  with $u_{n+1}(x)=0$. 
Then  $u_n(x)-l_n(x)=0$ so we have that  $u_n(x) = l_n(x) \neq 0 $ by the induction hypothesis. This implies 
$$l_{n+1}(x) = x u_n(x)+(1-x)l_n(x)=u_n(x) \neq 0 .$$
Therefore the claim is proved. 

By the claim, for each $n\in \NN$ and for each $x\in \RR,$ $p_n(x) = 0 $ if and only if $c_n(x) = 1/2$. Hence fix $n\in \NN$ and we consider the set $\overline{c_n}^{-1}(\{1/2\}).$ 
Then $\infty \notin \overline{c_n}^{-1}(\{1/2\})$ and so $\overline{c_n}^{-1}(\{1/2\})\subset \RR.$ Moreover since $\overline{c_n}$ is an $n$-fold covering map by Lemma  \ref{windingnumber}, the cardinality of the set $\overline{c_n}^{-1}(\{1/2\})$ is exactly $n.$ 
This implies that  there are $n$ distinct real roots of the polynomial $p_n.$ Since the polynomial  $p_n$ is of degree $n$, this completes the proof. 
\end{proof}

All roots of $p_n(\alpha)$ are positive. This follows from the fact that the polynomial $p_n(\alpha)$ has alternating coefficients. In other words, $$p_n(\alpha) = \sum_{k=0} ^n (-1)^{n-k} c_k \alpha^k $$ where $c_k \in \NN$.
Although it is quite surprising, we omit proof because it will not be needed.

 To verify that  the sequence $\{ \alpha_n \}$ is increasing, we need the following lemma.

\begin{lem} \label{less1/2}
 For all $n\in \NN$, $0 < (B_4 A^{-1})^n(0) < 1/2$.
\end{lem}
\begin{proof}
From the Jordan decomposition, we obtain $$B_{4} A^{-1} = \begin{bmatrix}
1 & -1 \\ 4 & -3
\end{bmatrix} = SJS^{-1}$$
 where
$$ 
J = \begin{bmatrix}
-1 & 1 \\ 0 & -1
\end{bmatrix} = - A^{-1} \textnormal{ and }
S = \begin{bmatrix}
1 & \frac{1}{2} \\ 2 & 0 
\end{bmatrix}.$$
Thus, for each $n\in \NN,$ we have that 
$$(B_4 A^{-1}) ^n = S J^n S^{-1} = (-1)^n \begin{bmatrix}
1-2n & n \\ -4n & 1+2n
\end{bmatrix}. $$
Hence, for each $n\in\NN,$  $(B_4 A^{-1})^n ( 0) = \frac{n}{2n+1} < \frac{1}{2}.$
\end{proof}

\begin{prop} \label{MaximalIncreasing}
 For each $n\in \NN,$ we denote the maximal root of $p_n$ by  $\alpha_n.$  Then $\alpha_n < \alpha_{n+1}$ for all $n\in \NN$.
\end{prop}
\begin{proof}
 We use mathematical induction on $n$.
 Since $p_1(\alpha)=\alpha - 3$ and $p_2(\alpha)=\alpha^2 -5 \alpha + 5$, the maximal roots are $3$ and $(5+\sqrt{5})/2,$ respectively. 
Since $3 < (5+\sqrt{5})/2$, our claim holds for $n = 1$. 
Now we assume that $\alpha_{i-1} <\alpha_i$ for all $i\in \{2, 3, \cdots, n\}.$ 
Then we want to show that $ \alpha_n < \alpha_{n+1}. $ 
As $\alpha_1 = 3$, $3<\alpha_n$ and since the real number $\alpha_n$ is a relation number, $\alpha_n < 4.$ Therefore  $3<\alpha_n < 4.$ On the other hand, as discussed in the proof of  Proposition \ref{rootreal},  $c_n(\alpha_n)=(B_{\alpha_n}A^{-1})^n \cdot 0 = 1/2$ and  
$$(B_{\alpha_n}A^{-1})^{n+1} ( 0) = \begin{bmatrix}
1 & -1 \\ \alpha_n & 1-\alpha_n
\end{bmatrix} \cdot \frac{1}{2}= \frac{1}{\alpha_n - 2}.$$ 
Hence $ 1/2 < (B_{\alpha_n}A^{-1})^{n+1}( 0) < 1$ as $3 < \alpha_n < 4.$ 
Also from Lemma \ref{less1/2}, $0<(B_4 A^{-1})^{n+1}( 0)<1/2.$
Therefore, $$0 < c_{n+1}(4) <\frac{1}{2} < c_{n+1}(\alpha_n) < 1. $$

Now we claim that  $1/2\in c_{n+1}([\alpha_n,4]).$  First since $\overline{c_{n+1}}$ is a $(n+1)$-fold covering  by Lemma \ref{windingnumber}, the set $\overline{c_{n+1}}^{-1}(\{c_{n+1}(4)\})$ has exactly $n+1$ distinct elements. 
Moreover, as $\overline{c_{n+1}}(\infty)=0$, $\infty \notin \overline{c_{n+1}}^{-1}(\{c_{n+1}(4)\})$ and so $\overline{c_{n+1}}^{-1}(\{c_{n+1}(4)\})\subset \RR$. 
Let $\beta_1$ be the smallest element of the intersection $[\alpha_n, 4]\cap \overline{c_{n+1}}^{-1}(\{c_{n+1}(4)\}).$ Similarly, we consider the set  $\overline{c_{n+1}}^{-1}(\{c_{n+1}(\alpha_n)\}).$ 
The set $\overline{c_{n+1}}^{-1}(\{c_{n+1}(\alpha_n)\})$ is also a finite set of real numbers. Let $\beta_0$ be the smallest element of the intersection of $\overline{c_{n+1}}^{-1}(\{c_{n+1}(\alpha_n)\})\cap [\alpha_n, \beta_1].$ 
Note that    $$0 < c_{n+1}(\beta_1) <\frac{1}{2} < c_{n+1}(\beta_0) < 1 $$ 
since $c_{n+1}(\beta_1)=c_{n+1}(4)$ and $c_{n+1}(\beta_0)=c_{n+1}(\alpha_n).$

Now we define $S_1 =\{ r \in \RR : c_{n+1}(4) < r <c_{n+1} (\alpha_n) \}$ and $S_2$ to be the interior of  $S^1 - S_1.$  Since $c_{n+1}$ is a covering map and is an open map, then the set $c_{n+1}((\beta_0,\beta_1))$ is open and connected. 
Moreover, by the choices of $\beta_0$ and $\beta_1$, $c_{n+1}((\beta_0,\beta_1) )\cap \{c_{n+1}(\beta_0), c_{n+1}(\beta_1)\}=\emptyset.$ 
Hence, one of the connected components of $S^1-\{c_{n+1}(\beta_0), c_{n+1}(\beta_1)\}$ contains the set $c_{n+1}((\beta_0,\beta_1))$, namely either  $c_{n+1}((\beta_0,\beta_1)) \subset S_1$ or $c_{n+1}((\beta_0,\beta_1))\subset S_2.$
Note that since  the interval $[\beta_0, \beta_1]$ is compact and connected, so is the image $c_{n+1}([\beta_0,\beta_1]) $. As $c_{n+1}(\beta_1)=c_{n+1}(4)$ and $c_{n+1}(\beta_0)=c_{n+1}(\alpha_n),$ either  $c_{n+1}([\beta_0,\beta_1])= \overline{S_1}$ or $c_{n+1}([\beta_0,\beta_1])= \overline{S_2}.$
\begin{figure}[h!]
\begin{center}
	\begin{tikzpicture}
		\draw [red,thick,domain=0:300] plot ({2*cos(\x)}, {2*sin(\x)});
		\draw [blue,thick,domain=300:340] plot ({2*cos(\x)}, {2*sin(\x)});
		\draw [red,thick,domain=340:360] plot ({2*cos(\x)}, {2*sin(\x)});
		\draw [black,fill] (0,-2) circle [radius=0.05] node [black,below] {$0$} -- (0,0) node [black,left] {$\ell$} -- (0,2) circle [radius=0.05] node [black,above] {$\infty$}; 
		\draw [black,fill] (2,0) circle [radius=0.05] node [black,right] {$1$};
		\draw [black,fill] (-2,0) circle [radius=0.05] node [black,left] {$-1$};
		\draw [black,fill] (1.87,-0.7) circle [radius=0.05] node [black,right] {$c_{n+1}(\alpha_n)$};
		\draw [black,fill] (1.05,-1.71) circle [radius=0.05] node [black,below] {$\ \ c_{n+1}(4)$};
		\draw [black,fill] (1.41,-1.41) circle [radius=0.05] node [black,above left] {$\frac{1}{2}$};
\end{tikzpicture}
\caption{The Candidates of the image of a closed interval $[\alpha_n,4]$ under the map $c_{n+1}$. 
The blue one is set $S_1$ and red one is $S_2$. The black vertical line is the geodesic $\ell$.
The blue arc and red arc meet at $c_{n+1}(\alpha_n)$ and $c_{n+1}(4)$.}
\end{center}
\end{figure}
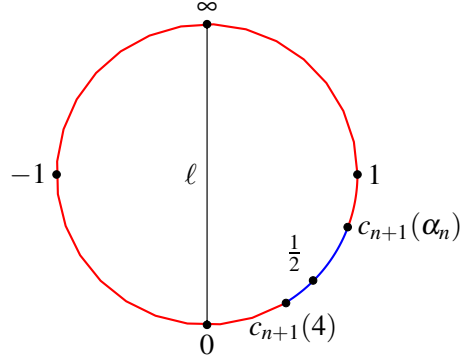
 To prove that the number $1/2$ belongs to the image $c_{n+1}([\beta_0,\beta_1])$, it suffices to show that the image can not be $\overline{S_2}$.
By Lemma \ref{rotationlemma}, $c_{n+1}$ is a clockwise map so $c_{n+1}$ is strictly decreasing at $\beta_0$ because $c_{n+1}(\beta_0) \neq \infty$.
 Thus, $c_{n+1}$ is strictly decreasing on the interval $[\beta_0,\beta_0+\epsilon_0]$ for some $\epsilon_0>0$.
Since $\beta_0 < 4$, we can choose $\epsilon_1$ so that $\beta_0 + \epsilon_1 < 4$. 
Moreover, we can choose $\epsilon_2$ such that $\epsilon_2 < c_{n+1}(\beta_0) - c_{n+1}(\beta_1)$.
Since $c_{n+1}$ is continuous, there exists a $\delta>0$ such that $$|c_{n+1}(\beta_0)-c_{n+1}(y)|<\epsilon_2$$ whenever $|\beta_0-y|<\delta$.
Let $\epsilon$ be the minimum of $\epsilon_0,$ $\epsilon_1,$ and $\delta$.
Then by the choice of $\epsilon$, the image $c_{n+1}([\beta_0, \beta_0 + \epsilon])$ must contain a point in $S_1$.
This shows that the only possible image of $[\beta_0,\beta_1]$ under $c_{n+1}$ is $\overline{S_1}.$  
Thus $1/2$ belongs to $c_{n+1}([\beta_0,\beta_1])$ and so to $c_{n+1}([\alpha_n,4]).$
 
 This implies that there exists $\alpha_{+}$ such that $\alpha_n < \alpha_{+} < 4$ and $p_{n+1}(\alpha_{+})=0$. 
Since $\alpha_{n+1}$ is the maximal root of $p_{n+1}(\alpha)$, we have $$\alpha_n < \alpha_{+} \leq \alpha_{n+1}. $$
\end{proof}

 To prove that the sequence $\{ \alpha _n \}_{n=1}^\infty$ converges to $4$, we use the notion of rotation numbers. 

\begin{lem} \label{denselemma}
Let $D= \left \{ \alpha\in (3,4) :  \{ (B_{\alpha} A^{-1})^n(0) : n \in \NN \} \text{ is \ dense \ in} \ S^1\right \}.$ The set $D$ is dense in the open interval $(3,4)$.
\end{lem}
\begin{proof}
 Recall that for each $\alpha\in (3,4),$ the matrix $$B_{\alpha} A^{-1} = \begin{bmatrix}
1 & -1 \\ \alpha & 1 - \alpha
\end{bmatrix} $$ is elliptic and the rotation number $\operatorname{rot}(B_{\alpha} A^{-1})$ is given by $\frac{1}{\pi}\cos^{-1} \left( {\frac{2-\alpha}{2}} \right ).$ 
Since a map from $(3,4)$ to $\RR$ defined by $\alpha \mapsto \frac{1}{\pi}\cos^{-1} \left ( {\frac{2 - \alpha}{2}} \right )$ is a homeomorphism, the set $$ \left\{ \alpha \in (3,4) : \frac{1}{\pi} \cos^{-1} \left ( {\frac{2-\alpha}{2}} \right ) \in \RR - \QQ \right\}$$
is dense in $(3,4)$. 
Since $B_{\alpha} A^{-1}$ is a smooth map for all $\alpha\in (3,4)$, by the classification of circle actions, we have that $\operatorname{rot}(B_{\alpha} A^{-1}) \in \RR - \QQ$ if and only if the set $\{ (B_{\alpha} A^{-1})^n(0) : n \in \ZZ \}$ is dense in $S^1$.
Moreover, as $\{ (B_{\alpha} A^{-1})^n ( 0) : n \in \ZZ \}$ is dense in $S^1$, the set $\{ (B_{\alpha} A^{-1})^n (0) : n \in \NN \}$ is dense in $S^1$.
\end{proof}

\vspace{0.6cm}

\begin{proof}[\textit{Proof of the Main Theorem \ref{themainthm}.}] 
\phantomsection
\label{lastproof}
 Choose any $\epsilon$ such that $0<\epsilon < 1$. To complete the proof, we have to prove that there exists $N \in \NN$ such that $ 4-\epsilon < \alpha_N $ since the sequence $\{ \alpha_n \}_{n=1}^\infty$ is increasing. 
By Lemma \ref{denselemma}, the set $ \left \{ (B_a A^{-1})^n ( 0) : n \in \NN \right \}$ is dense in $S^1$ for some real number $a$ with $4-\epsilon < a < 4$. Now choose $N$ so that $$ 1/2 < (B_a A^{-1})^N ( 0) < 1.$$
By Lemma \ref{less1/2}, we obtain $(B_4 A^{-1})^N \cdot 0 < 1/2$. Thus, the image of closed interval $[a,4]$ under $c_N$ must contain $1/2$ since the map $c_N$ is clockwise.
 Hence, this implies that there is a real number $a_0$ such that  $a<a_0<4$ and $(B_{a_0} A^{-1} )^N (0) = 1/2$. 
By the choice of $a_0$, the real number $a_0$ is one of the solutions of the polynomial $p_N.$ Since
$$ 4-\epsilon < a < a_0 \leq \alpha_N <4,$$ this completes the proof.
\end{proof}

\bibliography{biblio}
\bibliographystyle{abbrv}

\end{document}